\documentclass[11pt]{article}
\usepackage{t1enc}
\usepackage[latin1]{inputenc}
\usepackage[english]{babel}
\usepackage{amsmath,amsthm}
\usepackage{amsfonts}
\usepackage{latexsym}
\usepackage[dvips]{graphicx}
\usepackage{graphicx}
\usepackage{mathrsfs}
\title{ The relation between the independence number and rank of a signed graph}
\author{ Shengjie  He$^1$, Rong-Xia Hao$^1$\footnote{Corresponding author.
Emails: he1046436120@126.com (Shengjie  He), rxhao@bjtu.edu.cn (Rong-Xia Hao)}\\
{\small\em 1. Department of Mathematics, Beijing Jiaotong University, Beijing,
100044, China}\\
  }

\date{} \textwidth 16cm \textheight 22cm \topmargin 0 cm \hoffset
-1.5 cm \voffset 0cm

\newtheorem{theorem}{Theorem}[section]
\newtheorem{lemma}[theorem]{Lemma}

\newtheorem{corollary}[theorem]{Corollary}

\begin{document}
\baselineskip 0.50cm \maketitle

\begin{abstract}
A signed graph $(G, \sigma)$ is a graph with a sign attached to each of its
edges, where $G$ is the underlying graph of $(G, \sigma)$.
Let $c(G)$,  $\alpha(G)$ and $r(G, \sigma)$ be the cyclomatic number,
the independence number and the rank of the adjacency matrix of $(G, \sigma)$, respectively.
In this paper, we study the relation among the independence number, the rank
and the cyclomatic number of a signed graph $(G, \sigma)$ with order $n$, and prove that $2n-2c(G) \leq r(G, \sigma)+2\alpha(G) \leq 2n$. Furthermore, the signed graphs that reaching the lower bound are
investigated.


\end{abstract}

\section{Introduction}

All graphs considered in this paper are simple and finite.
Let $G=(V, E)$ be an undirected graph with $V=\{ v_{1}, v_{2}, \cdots, v_{n} \}$ is the vertex set and $E$ is the edge set. For a vertex $u \in V(G)$, the $degree$ of $u$, denote by $d_{G}(u)$, is the number of vertices
which are adjacent to $u$. A vertex of $G$ is called a {\it pendant vertex} if it is a vertex of degree one in $G$, whereas a vertex of $G$ is called a {\it quasi-pendant vertex} if it is adjacent to
a pendant vertex in $G$ unless it is a pendant vertex.
Denote by $P_n$, $S_n$ and $C_n$ a path, star and cycle on $n$ vertices, respectively.
The {\it adjacency matrix} $A(G)$ of $G$ is an $n \times n$ matrix whose $(i, j)$-entry equals to
1 if vertices $v_{i}$ and $v_j$ are adjacent and 0 otherwise.
We refer to \cite{BONDY} for undefined terminologies and notation.

A {\it signed graph} $(G, \sigma)$ consists of a simple graph $G=(V, E)$, referred to as its underlying graph,
and a mapping $\sigma: E  \rightarrow  \{ +, - \}$, its edge labelling. To avoid confusion, we often
write $V(G)$ and $E(G)$ for $V(G, \sigma)$ and $E(G, \sigma)$, respectively.
The adjacency matrix of $(G, \sigma)$ is $A(G, \sigma)=(a_{ij}^{\sigma})$ with $a_{ij}^{\sigma}=\sigma(v_iv_j)a_{ij}$, where $(a_{ij})$ is the
adjacent matrix of the underlying graph $G$.
Let $(G, \sigma)$ be a signed graph. An edge $e'$ is said to be positive or negative if $\sigma(e')=+$ or $\sigma(e')=-$, respectively.
In the case of $\sigma =+$, which is an all-positive
edge labelling, $A(G, +)$ is exactly the classical adjacency matrix of $G$.
Thus a simple graph can always be viewed as a singed graph with all positive edges.
Let $C$ be a cycle of $(G, \sigma)$. The $sign$ of $C$ is defined by $\sigma(C)=\prod_{e \in C}\sigma(e)$.
A cycle $C$ is said to be positive or negative if $\sigma(C)=+$ or $\sigma(C)=-$, respectively.
By definition, a  cycle $C$ is positive if and only if it has even number of negative edges.
The $rank$ of a signed graph $(G, \sigma)$,
written as $r(G, \sigma)$, is defined to be the rank of its adjacency matrix $A(G, \sigma)$.
The $nullity$
of $(G, \sigma)$ is the multiplicity of the zero eigenvalues of $A(G, \sigma)$.

A subset $I$ of $V(G)$ is called an $independent$ $set$ if
any two vertices of $I$ are independent in a graph $G$. The $independence$ $number$ of $G$, denoted by
$\alpha(G)$, is the number of vertices in a maximum independent set of $G$.
Let $c(G)$ be the {\it cyclomatic number} of a graph $G$,
that is $c(G)=|E(G)|-|V(G)|+\omega(G)$, where $\omega(G)$ is the number of connected components of $G$.
The $matching$ $number$ of $G$, denoted by $m(G)$,
is the cardinality of a maximum matching of $G$.
For a signed graph $(G, \sigma)$, the independence number, cyclomatic number
and matching number of $(G, \sigma)$  are defined to be the independence number, cyclomatic number and
matching number of its underlying graph, respectively.

If $G$ is a graph which any two cycles (if any) of $G$ have no vertices in common. Denoted by $\mathscr{C}_{G}$ the set of all cycles of $G$. Contracting each cycle of $G$ into a vertex (called a {\it cyclic vertex}), we obtain a forest denoted by $T_{G}$.
Denoted by $\mathscr{O}_{G}$ the set of all cyclic vertex of $G$.
Moreover, denoted by $[T_{G}]$ the subgraph of $T_{G}$
induced by all non-cyclic vertices. It is obviously that $[T_{G}]=T_{G}-\mathscr{O}_{G}$.

The study on rank and nullity of graphs is a major and heated issue in graph theory.
In \cite{WANGLONG}, the relation between the independence number and the rank of a graph was investigated by Wang and Wong. Gutman et al. \cite{GUT} researched the nullity of line graphs of trees.
Mohar et al. characterized the properties of the $H$-rank of mixed graphs in \cite{MOHAR}.
Li et at. studied the relation among the rank, the matching number and the cyclomatic number of
oriented graphs and mixed graphs in \cite{HLSC} and \cite{LSC}, respectively.
Bevis et al. \cite{BEVI} obtained some results examining several cases of vertex addition.
In \cite{MAH2}, Ma et al. researched the relation among the nullity, the dimension of cycle space and the number of pendant vertices of a graph.

In recent years, the study of the rank and nullity of signed graphs received increased attention.
The rank of signed planar graphs was investigated by Tian et al. \cite{TWDY}.
Belardo et al. studied the Laplacian spectral of signed graphs in \cite{BELA}.
You et al. \cite{YLH} characterized the nullity of signed graphs.
In \cite{WSJ}, the relation between the rank of a signed graph and the rank of its underlying graph was
researched by Wang.
The nullity of unicyclic signed graphs and bicyclic signed graphs were studied by
Fan et al. in \cite{FYZ} and \cite{FYZD}, respectively.
Wong et al. \cite{WANGXIN} characterized the positive inertia index of the signed graphs.
He et al. \cite{HSJ} studied the relation among the matching number, the cyclomatic number and the rank of
the signed graphs.
For other research of the rank of a graph one may be referred to those in \cite{CVE,LEE,ROB,HOU,WDY}.

In this paper, the relation among the rank of a signed graph $(G, \sigma)$
and the cyclomatic number and the independence number of its underlying graph is investigated.
We prove that $2n-2c(G) \leq r(G, \sigma)+2\alpha(G) \leq 2n$ for any signed graph $(G, \sigma)$ with order $n$. Moreover,  the extremal graphs which attended the lower bound are characterized.
Our main results are the following Theorems \ref{T30} and \ref{T50}.

\begin{theorem}\label{T30}
Let $(G, \sigma)$ be a signed graph with order $n$. Then
$$2n-2c(G) \leq r(G, \sigma)+2\alpha(G) \leq 2n.$$
\end{theorem}

\begin{theorem}\label{T50}
Let $(G, \sigma)$ be a signed graph with order $n$. Then $ r(G, \sigma)+2\alpha(G) = 2n-2c(G)$
if and only if all the following conditions hold for $(G, \sigma)$:

{\em(i)} the cycles (if any) of $(G, \sigma)$ are pairwise vertex-disjoint;

{\em(ii)} for each cycle (if any) $C_{q}$ of $(G, \sigma)$, either $ q \equiv 0 \ (\rm{mod} \ 4)$ and $\sigma(C_{q})=+$ or $ q \equiv 2 \ (\rm{mod} \ 4)$ and $\sigma(C_{q})=-$;

{\em(iii)} $\alpha(T_{G})=\alpha([T_{G}])+c(G)$.
\end{theorem}

The rest of this paper is organized as follows. Prior to showing our main results, in Section 2, some elementary notations and some useful lemmas are established. In Section 3, we give the proof of the main result of this paper. In Section 4, the extremal signed graphs which attained the lower bound of Theorem \ref{T30} are characterized.

\section{Preliminaries}

In this section, some useful lemmas which will be
used in the proofs of our main results are presented.

For $X \subseteq V(G)$, $G-X$ is the induced subgraph obtained from $G$ by deleting all vertices in $X$ and
all incident edges. In particular, $G-\{ x \}$ is usually written as $G-x$ for simplicity.
For an induced subgraph $H$ and a vertex $u$
outside $H$, the induced subgraph of $G$ with vertex set $V(H) \cup \{ u \}$ is simply written as $H+u$.

\begin{lemma} \label{L16}{\rm\cite{WSJ}}
Let $(G, \sigma)$ be a signed graph.

{\em(i)} If $(H, \sigma)$ is an induced subgraph of $(G, \sigma)$, then $r(H, \sigma) \leq r(G, \sigma)$.

{\em(ii)} If $(G_{1}, \sigma), (G_{2}, \sigma), \cdots, (G_{t}, \sigma)$ are all the connected components of $(G, \sigma)$, then $r(G, \sigma)=\sum_{i=1}^{t}r(G_{i}, \sigma)$.

{\em(iii)} $r(G, \sigma)\geq 0$ with equality if and only if $(G, \sigma)$ is an empty graph.
\end{lemma}

\begin{lemma} \label{L13}{\rm\cite{FYZ}}
Let $y$ be a pendant vertex of $(G, \sigma)$ and $x$ is the neighbour of $y$.
Then $r(G, \sigma)=r((G, \sigma)-  \{ x, y \} )+2$.
\end{lemma}

\begin{lemma} \label{L14}{\rm\cite{BEVI}}
Let $x$ be a vertex of $(G, \sigma)$.
Then $r(G, \sigma)-2 \leq r((G, \sigma)-x) \leq r(G, \sigma)$.
\end{lemma}

\begin{lemma} \label{L15}{\rm\cite{CVE}}
Let $(T, \sigma)$ be a signed acyclic graph.
Then $r(T, \sigma)= r(T)=2m(T)$.
\end{lemma}

\begin{lemma} \label{L051}{\rm\cite{BONDY}}
Let $T$ be a bipartite graph with order $n$.
Then $\alpha(T)+m(T)=n$.
\end{lemma}

It is obviously that by Lemmas \ref{L15} and \ref{L051}, we have

\begin{lemma} \label{L052}
Let $(T, \sigma)$ be a signed acyclic graph.
Then $r(T, \sigma)+2\alpha(T)=2n$.
\end{lemma}

\begin{lemma} \label{L053}{\rm\cite{HLSC}}
Let $y$ be a pendant vertex of a graph $G$ and $x$ is the neighbour of $y$.
Then $\alpha(G)=\alpha(G-x)=\alpha(G-\{ x, y \})+1$.
\end{lemma}

\begin{lemma} \label{L12}{\rm\cite{YLH}}
Let $(C_{n}, \sigma)$ be a signed cycle of order $n$. Then
$$r(C_{n}, \sigma)=\left\{
             \begin{array}{ll}
               n, & \hbox{if $n$ is odd;} \\
               n, & \hbox{if $n \equiv 0 \ (\rm{mod} \ 4)$ and $\sigma(C_{n})=-$;} \\
               n, & \hbox{if $n \equiv 2 \ (\rm{mod} \ 4)$ and $\sigma(C_{n})=+$;} \\
               n-2, & \hbox{if $n \equiv 0 \ (\rm{mod} \ 4)$ and $\sigma(C_{n})=+$;} \\
               n-2, & \hbox{if $n \equiv 2 \ (\rm{mod} \ 4)$ and $\sigma(C_{n})=-$.}
             \end{array}
           \right.
$$
\end{lemma}

\begin{lemma} \label{L23}{\rm\cite{WDY}}
Let $G$ be a graph with $x \in V(G)$.

{\em(i)} $c(G)=c(G-x)$ if $x$ lies outside any cycle of $G$;

{\em(ii)} $c(G-x) \leq c(G)-1$ if $x$ lies on a cycle of $G$;

{\em(iii)} $c(G-x) \leq c(G)-2$ if $x$ is a common vertex of distinct cycles of $G$.
\end{lemma}

\begin{lemma} \label{L054}{\rm\cite{HLSC}}
Let $G$ be a graph. Then

{\em(i)} $\alpha(G)-1  \leq \alpha(G-x) \leq \alpha(G) $ for any vertex $x \in V(G)$;

{\em(ii)} $\alpha(G-e) \geq \alpha(G)$ for any edge $e \in E(G)$.
\end{lemma}

\begin{lemma} \label{L055}{\rm\cite{HLSC}}
Let $T$ be a tree with at least one edge and $T_{0}$ be the subtree obtained
from $T$ by deleting all pendant vertices of $T$.

{\em(i)} $\alpha(T)  \leq \alpha(T_{0}) +p(T) $, where $p(T)$ is the number of pendent vertices of $T$;

{\em(ii)} If $\alpha(T) = \alpha(T-D)+|D|$ for a subset $D$ of $V(T)$, then there is a pendant vertex $x$ such
that $x \notin D$.
\end{lemma}

\section{Proof of Theorem \ref{T30}}

In this section, we study the relation among the rank
and the independence number and the cyclomatic number of a signed graph, and give the proof for Theorem \ref{T30}.

\noindent
{\bf The proof of Theorem \ref{T30}.}

First, we prove the inequality on the left of Theorem \ref{T30}.
We argue by induction on $c(G)$ to show that $2n-2c(G) \leq r(G, \sigma)+2\alpha(G) $.
If $c(G)=0$, then $(G, \sigma)$ is a signed tree, and so result follows from Lemma \ref{L052}.
Hence we assume that $c(G) \geq 1$. Let
$u$ be a vertex on some cycle of $(G, \sigma)$ and $(G', \sigma)=(G, \sigma)-u$.
Let $(G_{1}, \sigma), (G_{2}, \sigma), \cdots, (G_{l}, \sigma)$ be all
connected components of $(G', \sigma)$.
By Lemma \ref{L23}, we have
\begin{equation} \label{E1}
\sum_{i=1}^{l}c(G_{i})=c(G') \leq c(G)-1.
\end{equation}
By the induction hypothesis, one has
\begin{equation} \label{E2}
2(n-1)-2c(G') \leq r(G', \sigma)+2\alpha(G')
\end{equation}
By Lemmas \ref{L054} and \ref{L14}, we have
\begin{equation} \label{E3}
\sum_{i=1}^{l}\alpha(G_{i})=\alpha(G') \leq \alpha(G),
\end{equation}
and
\begin{equation} \label{E4}
\sum\limits_{i=1}^{l}r(G_{i}, \sigma)=r(G', \sigma) \leq r(G, \sigma).
\end{equation}
Thus the desired inequality now follows by combining (\ref{E1}), (\ref{E2}),  (\ref{E3}) and (\ref{E4}),
\begin{eqnarray} \label{E0}
r(G, \sigma)+2\alpha(G) & \ge & r(G', \sigma)+2\alpha(G')
\\ \nonumber
& \ge & 2(n-1)-2c(G')
\\ \nonumber
& \geq & 2(n - 1) - 2(c(G) - 1) = 2n-2c(G),
\end{eqnarray}
as desired.

Next,  we show that $ r(G, \sigma)+2\alpha(G) \leq 2n $.
Let $I$ be a maximum independence set of $G$, i.e., $|I|=\alpha(G)$. Then

$$A(G, \sigma)=\left(
                \begin{array}{cc}
                  \mathbf{0} & \mathbf{B} \\
                  \mathbf{B}^{\top} & \mathbf{A} \\
                \end{array}
              \right)
$$
where $\mathbf{B}$ is a matrix of $A(G, \sigma)$ with row indexed by $I$ and column indexed by $V(G)-I$,
$\mathbf{B}^{\top}$ refers to the transpose of $\mathbf{B}$ and $\mathbf{A}$ is the adjacency matrix of the induced subgraph $G-I$. Then it can be checked that
$$r(G, \sigma)\leq r(\mathbf{0}, \mathbf{B})+r(\mathbf{B}^{\top}, \mathbf{A})\leq n-\alpha(G)+n-\alpha(G)=2n-2\alpha(G).$$
Thus,
$$r(G, \sigma)+2\alpha(G)\leq 2n.$$

This completes the proof of Theorem \ref{T30}.
$\square$

A signed graph $(G, \sigma)$ with $r(G, \sigma)+2\alpha(G)= 2n-2c(G)$ is called a {\bf lower-optimal}
signed graph. One can utilize the arguments above to make the following observations.

\begin{corollary} \label{C31}
Let $u$ be a vertex of $(G, \sigma)$ lying on a signed cycle. If $r(G, \sigma)  = 2m(G)-2c(G)$, then
each of the following holds.
\\
{\em(i)} $r(G, \sigma)=r((G, \sigma)-u)$;
\\
{\em(ii)} $(G, \sigma)-u$ is lower-optimal;
\\
{\em(iii)} $c(G)=c(G-u)+1$;
\\
{\em(iv)} $\alpha(G)=\alpha(G-u)$;
\\
{\em(v)} $u$ lies on just one signed cycle of $(G, \sigma)$ and $u$ is not a quasi-pendant vertex of $(G, \sigma)$.
\end{corollary}
\begin{proof}
In the proof arguments of Theorem \ref{T30} that justifies $r(G, \sigma)+2\alpha(G) \geq 2n-2c(G)$.
If both ends of (\ref{E0}) are the same, then all inequalities in (\ref{E0}) must be
equalities, and so Corollary \ref{C31} (i)-(iv) are observed.

To show (v). By Corollary \ref{C31} (iii) and Lemma \ref{L23}, we conclude that $u$ lies on just one
signed cycle of $(G, \sigma)$.
Suppose to the contrary that $u$ is a quasi-pendant vertex which adjacent to a pendant vertex $v$.
Then by Lemma \ref{L13}, we have
$$r((G, \sigma)-u)=r((G, \sigma)- \{ u, v  \})=r(G, \sigma)-2,$$
which is a contradiction to (i). This completes the proof of the corollary.
\end{proof}

\section{Proof of Theorem \ref{T50}.}

Recall that a signed graph $(G, \sigma)$ with order $n$ is {\it lower-optimal} if $r(G, \sigma)+2\alpha(G)=2n-2c(G)$,
or equivalently, the signed graph which attain the lower bound in Theorem \ref{T30}.
In this section, we introduce some lemmas firstly, and then we give the proof for Theorem \ref{T50}.
By Lemma \ref{L12}, the following Lemma \ref{L50} can be obtained directly.

\begin{lemma} \label{L50}
The signed cycle $(C_{q}, \sigma)$ is lower-optimal if and only if either $ q \equiv 0 \ (\rm{mod} \ 4)$ and $\sigma(C_{q})=+$ or $ q\equiv 2 \ (\rm{mod} \ 4)$ and $\sigma(C_{q})=-$.
\end{lemma}

\begin{lemma} \label{L51}
Let $(G, \sigma)$ be a signed graph and $(G_{1}, \sigma), (G_{2}, \sigma), \cdots, (G_{k}, \sigma)$ be all connected components of $(G, \sigma)$. Then $(G, \sigma)$ is lower-optimal if and only if $(G_{j}, \sigma)$ is lower-optimal for
each $j \in \{1, 2, \cdots, k \}$.
\end{lemma}
\begin{proof} (Sufficiency.) For each $i \in \{ 1, 2, \cdots, k \}$, one has that
$$r(G_{i}, \sigma)+2\alpha(G_{i})= 2|V(G_{i})|-2c(G_{i}).$$
Then, one can get $(G, \sigma)$ is lower-optimal immediately follows from the fact that
$$r(G, \sigma)=\sum\limits_{j=1}^{k}r(G_{j}, \sigma),$$
$$\alpha(G)=\sum\limits_{j=1}^{k}\alpha(G_{j}),$$
and
$$c(G)=\sum\limits_{j=1}^{k}c(G_{j}).$$

(Necessity.) Suppose to the contrary that there is a connected component of $(G, \sigma)$, say $(G_{1}, \sigma)$, which is not lower-optimal. Then
$$r(G_{1}, \sigma)+2\alpha(G_{1}) > 2|V(G_{1})|-2c(G_{1}),$$
and by Theorem \ref{T30}, for each $ j \in \{ 2, 3, \cdots, k \}$, we have
$$r(G_{j}, \sigma) +2\alpha(G_{j}) \geq 2|V(G_{j})|-2c(G_{j}).$$
Thus, one has that
$$r(G, \sigma)+2\alpha(G) > 2|V(G)|-2c(G),$$
a contradiction.
\end{proof}

\begin{lemma} \label{L56}
Let $u$ be a pendant vertex of a signed graph $(G, \sigma)$ and $v$ be the vertex which adjacent to $u$. Let $(G_{0}, \sigma)=(G, \sigma)-\{ u, v \}$. Then, $(G, \sigma)$ is lower-optimal if and only if $v$ is not on any signed cycle of $(G, \sigma)$ and $(G_{0}, \sigma)$ is lower-optimal.

\end{lemma}
\begin{proof}
(Sufficiency.) If $v$ is not on any signed cycle, by Lemma \ref{L23}, we have
$$c(G)=c(G_{0}).$$
By Lemmas \ref{L13} and \ref{L053}, one has that
$$r(G, \sigma)=r(G_{0}, \sigma)+2, \alpha(G)=\alpha(G_{0})+1.$$
Thus, one can get $(G, \sigma)$ is lower-optimal by the condition that
$$r(G_{0}, \sigma)+2\alpha(G_{0})= 2|V(G_{0})|-2c(G_{0}).$$

(Necessity.) By Lemma \ref{L13} and Corollary \ref{C31} and the condition that $(G, \sigma)$ is lower-optimal,
it can be checked that
$$r(G_{0}, \sigma)+2\alpha(G_{0})=2|V(G_{0})|-2c(G).$$
It follows from Theorem \ref{T30} that one has
$$r(G_{0}, \sigma)+2\alpha(G_{0}) \geq 2|V(G_{0})|-2c(G_{0}).$$
By the fact that $c(G_{0}) \leq c(G)$, then we have
$$c(G)=c(G_{0}), r(G_{0}, \sigma)+2\alpha(G_{0})=2|V(G_{0})|-2c(G_{0}).$$
Thus $(G_{0}, \sigma)$ is also lower-optimal and so the lemma is justified.
\end{proof}

\begin{lemma} \label{L55}
Let $(G, \sigma)$ be a signed graph obtained by joining a vertex $x$ of a signed cycle $C_{l}$ by
a signed edge to a vertex $y$ of a signed connected graph $(K, \sigma)$. If $(G, \sigma)$ is lower-optimal, then the following properties hold for $(G, \sigma)$.

{\em(i)} Either $ l \equiv 0 \ (\rm{mod} \ 4)$ and $\sigma(C_{l})=+$ or $ l\equiv 2 \ (\rm{mod} \ 4)$ and $\sigma(C_{l})=-$;

{\em(ii)} $r(G, \sigma)=l-2+r(K, \sigma)$, $\alpha(G)=\frac{l}{2}+\alpha(K)$;

{\em(iii)} $(K, \sigma)$ is lower-optimal;

{\em(iv)} Let $(G', \sigma)$ be the induced signed subgraph of $(G, \sigma)$ with vertex set $V(K)\cup \{ x\}$. Then $(G', \sigma)$ is also lower-optimal;

{\em(v)} $\alpha(G')=\alpha(K)+1$ and $r(G', \sigma)=r(K, \sigma)$.
\end{lemma}
\begin{proof}
{\bf  (i):} We show (i) by induction on the order $n$ of $(G, \sigma)$. By Corollary \ref{C31}, $x$ can not
be a quasi-pendant vertex of $(G, \sigma)$, then $y$ is not an isolated vertex of $(G, \sigma)$. Then, $(K, \sigma)$ contains at least two vertices, i.e., $n\geq l+2$. If $n=l+2$, then $(K, \sigma)$ contains exactly
two vertices, without loss of generality, assume them be $y$ and $z$. Thus, one has that $(C_{l}, \sigma)=(G, \sigma)-\{ y, z \}$. By Lemma \ref{L56}, we have $(C_{l}, \sigma)$ is lower-optimal.
Then (i) follows from Lemma \ref{L50} directly.

Next, we consider the case of $n \geq l+3$. Suppose that (i) holds for every lower-optimal signed graph with order smaller than $n$. If $(K, \sigma)$ is a forest. Then $(G, \sigma)$ contains at least one isolated vertex. Let $u$ be
a pendant vertex of $(G, \sigma)$ and $v$ be the vertex which adjacent to $u$. By Corollary \ref{C31}, $v$ is not on $(C_{l}, \sigma)$. By Lemma \ref{L56}, one has that $(G, \sigma)-\{ u, v \}$ is lower-optimal.
By induction hypothesis to $(G, \sigma)-\{ u, v \}$, we have either $ l \equiv 0 \ (\rm{mod} \ 4)$ and $\sigma(C_{l})=+$ or $ l\equiv 2 \ (\rm{mod} \ 4)$ and $\sigma(C_{l})=-$. If $(K, \sigma)$ contains cycles. Let $g$ be a vertex lying on a cycle of $(K, \sigma)$. By Corollary \ref{C31}, $(G, \sigma)-g$ is lower-optimal. Then, the induction hypothesis to $(G, \sigma)-g$ implies that either $ l \equiv 0 \ (\rm{mod} \ 4)$ and $\sigma(C_{l})=+$ or $ l\equiv 2 \ (\rm{mod} \ 4)$ and $\sigma(C_{l})=-$. This completes the proof of (i).

{\bf  (ii):} Since $x$ lies on a cycle of $(G, \sigma)$, by Corollary \ref{C31} and Lemmas \ref{L13} and \ref{L053}, one has that
\begin{equation} \label{E6}
r(G, \sigma)=r((G, \sigma)-x)= r(P_{l-1}, \sigma)+r(K, \sigma)=l-2+r(K, \sigma),
\end{equation}
and
\begin{equation} \label{E7}
\alpha(G)=\alpha(G-x)= \alpha(P_{l-1})+\alpha(K)=\frac{l}{2}+\alpha(K).
\end{equation}

{\bf  (iii):} As $C_{l}$ is a pendant cycle of $(G, \sigma)$, one has that
\begin{equation} \label{E8}
c(K)=c(G)-1.
\end{equation}
By (\ref{E6})-(\ref{E8}), we have
\begin{equation} \label{E9}
r(K, \sigma)+2\alpha(K)=2(n-l)-2c(K).
\end{equation}

{\bf  (iv):} Let $s$ be a vertex of $(C_{l}, \sigma)$ which adjacent to $x$. Then, by  Corollary \ref{C31} and Lemmas \ref{L13} and \ref{L053}, we have
\begin{equation} \label{E10}
r(G, \sigma)=r((G, \sigma)-s)= l-2+r(G', \sigma),
\end{equation}
and
\begin{equation} \label{E11}
\alpha(G)=\alpha(G-s)= \frac{l-2}{2}+\alpha(G').
\end{equation}
Moreover,it is obviously that
\begin{equation} \label{E12}
c(G)=c(G')+1.
\end{equation}
From (\ref{E10})-(\ref{E12}), we have
\begin{eqnarray*}
r(G', \sigma)+2\alpha(G')&=&r(G, \sigma)+2\alpha(G)-2(l-2)\\
&=&2n-2c(G)-2(l-2)\\
&=&2(n-l+1)-2c(G').
\end{eqnarray*}

{\bf  (v):} Combining (\ref{E6}) and (\ref{E10}), one has that
$$r(K, \sigma)=r(G', \sigma).$$
From (\ref{E7}) and (\ref{E11}), we have
$$\alpha(K)+1=\alpha(G').$$

This completes the proof.
\end{proof}

\begin{lemma} \label{L58}
Let $(G, \sigma)$ be a lower-optimal signed graph. Then the following properties hold for $(G, \sigma)$.

{\em(i)} The cycles (if any) of $(G, \sigma)$ are pairwise vertex-disjoint;

{\em(ii)} For each cycle $C_{l}$ of $(G, \sigma)$, either $ l \equiv 0 \ (\rm{mod} \ 4)$ and $\sigma(C_{l})=+$ or $ l\equiv 2 \ (\rm{mod} \ 4)$ and $\sigma(C_{l})=-$;

{\em(iii)} $\alpha(G)=\alpha(T_{G})+\sum_{C \in \mathscr{C}_{G}}\frac{|V(C)|}{2}-c(G)$.
\end{lemma}

\begin{proof}
{\bf  (i):} By Corollary \ref{C31}, (i) follows directly.

{\bf  (ii)-(iii):} We argue by induction on the order $n$ of $G$ to show (ii)-(iii). If $n=1$, then (ii)-(iii)
hold trivially. Next, we consider the case of $n \geq 2$. Suppose that (ii)-(iii) holds for every lower-optimal signed graph with order smaller than $n$.

If $E(T_{G})=0$, i.e., $T_{G}$ is an empty graph, then each component of $(G, \sigma)$ is a cycle or an isolated vertex. By Lemmas \ref{L50} and \ref{L51}, we have either $ l \equiv 0 \ (\rm{mod} \ 4)$ and $\sigma(C_{l})=+$ or $ l\equiv 2 \ (\rm{mod} \ 4)$ and $\sigma(C_{l})=-$.  For each cycle $C_{l}$, it is routine to check that
$\alpha(C_{l})= \lfloor \frac{l}{2} \rfloor$.  Then (ii) and (iii) follows.

If $E(T_{G}) \geq 1$. Then $T_{G}$ contains at least one pendant vertex, say $x$. If $x$ is also a pendant vertex in $(G, \sigma)$, then $(G, \sigma)$ contains a pendant vertex. If $x$ is a vertex obtained by contracting a cycle of $(G, \sigma)$, then $(G, \sigma)$ contains a pendant cycle.
Then we will deal with the following two cases.

{\bf Case 1.} If $x$ is also a pendant vertex in $(G, \sigma)$. Let $y$ be the unique neighbour of $x$ and
$(G_{0}, \sigma)= (G, \sigma)-\{ x, y \}$. By Lemma \ref{L56}, one has that $y$ is not on any cycle of $(G, \sigma)$
and $(G_{0}, \sigma)$ is lower-optimal. By induction hypothesis, we have

(a) For each cycle $C_{l}$ of $(G_{0}, \sigma)$, either $ l \equiv 0 \ (\rm{mod} \ 4)$ and $\sigma(C_{l})=+$ or $ l\equiv 2 \ (\rm{mod} \ 4)$ and $\sigma(C_{l})=-$;

(b) $\alpha(G_{0})=\alpha(T_{G_{0}})+\sum_{C \in \mathscr{C}_{G_{0}}}\frac{|V(C)|}{2}-c(G_{0})$.

It is routine to check that all cycles of $(G, \sigma)$ are also in $(G_{0}, \sigma)$. Then for each cycle $C_{l}$ of $(G, \sigma)$, either $ l \equiv 0 \ (\rm{mod} \ 4)$ and $\sigma(C_{l})=+$ or $ l\equiv 2 \ (\rm{mod} \ 4)$ and $\sigma(C_{l})=-$. Thus (ii) follows. Furthermore, one has that
$$c(G)=c(G_{0}).$$

Sine $x$ is a pendant vertex of $(G, \sigma)$ and $y$ is a quasi-pendant vertex which is not in any cycle of $(G, \sigma)$, $x$ is a pendant vertex of $T_{G}$ and $y$ is a quasi-pendant vertex of $T_{G}$. Moreover,
$T_{G_{0}}=T_{G}-\{ x, y \}$. Thus, by Lemma \ref{L053} and assertion (b), we have
\begin{eqnarray*}
\alpha(G)&=&\alpha(G_{0})+1\\
&=&\alpha(T_{G_{0}})+\sum_{C \in \mathscr{C}_{G_{0}}}\frac{|V(C)|}{2}-c(G_{0})+1\\
&=&\alpha(T_{G})-1+\sum_{C \in \mathscr{C}_{G_{0}}}\frac{|V(C)|}{2}-c(G_{0})+1\\
&=&\alpha(T_{G})+\sum_{C \in \mathscr{C}_{G}}\frac{|V(C)|}{2}-c(G).
\end{eqnarray*}
Thus, (iii) holds in this case.

{\bf Case 2.} $(G, \sigma)$ contains a pendant cycle, say $C_{q}$.

In this case, let $x$ be the unique vertex of $C_{q}$ of degree 3. Let $K=G-C_{q}$ and $(G_{1}, \sigma)$ be the induced signed subgraph of $(G, \sigma)$ with vertex set $V(K)\cup \{ x\}$. By Lemma \ref{L55} (iv), one has that $(G_{1}, \sigma)$ is lower-optimal. By induction hypothesis, we have

(c) For each cycle $C_{l}$ of $(G_{1}, \sigma)$, either $ l \equiv 0 \ (\rm{mod} \ 4)$ and $\sigma(C_{l})=+$ or $ l\equiv 2 \ (\rm{mod} \ 4)$ and $\sigma(C_{l})=-$;

(d) $\alpha(G_{1})=\alpha(T_{G_{1}})+\sum_{C \in \mathscr{C}_{G_{1}}}\frac{|V(C)|}{2}-c(G_{1})$.

By Lemma \ref{L55} (i), we have either $ q \equiv 0 \ (\rm{mod} \ 4)$ and $\sigma(C_{q})=+$ or $ q\equiv 2 \ (\rm{mod} \ 4)$ and $\sigma(C_{q})=-$. Thus, (ii) follows from
$$\mathscr{C}_{G}=\mathscr{C}_{G_{1}} \cup C_{q}=\mathscr{C}_{K} \cup C_{q}.$$
Moreover, one has that
\begin{equation} \label{E13}
\sum_{C \in \mathscr{C}_{G}}\frac{|V(C)|}{2}=\sum_{C \in \mathscr{C}_{G_{1}}}\frac{|V(C)|}{2}+\frac{q}{2}=\sum_{C \in \mathscr{C}_{K}}\frac{|V(C)|}{2}+\frac{q}{2}.
\end{equation}

Since $C_{q}$ is a pendant cycle of $(G, \sigma)$, it is obviously that
\begin{equation} \label{E14}
c(G_{1})=c(K)=c(G)-1.
\end{equation}
By Lemma \ref{L55} (v), one has that
\begin{equation} \label{E15}
\alpha(G_{1})=\alpha(K)+1.
\end{equation}
Note that
\begin{equation} \label{E150}
T_{G_{1}}=T_{G}.
\end{equation}
By Lemma \ref{L55} (ii) and (\ref{E13})-(\ref{E150}), one has that
\begin{eqnarray*}
\alpha(G)&=&\alpha(K)+\frac{p}{2}\\
&=&\alpha(G_{1})+\frac{p}{2}-1\\
&=&\alpha(T_{G_{1}})+\sum_{C \in \mathscr{C}_{G_{1}}}\frac{|V(C)|}{2}-c(G_{1})+\frac{p}{2}-1\\
&=&\alpha(T_{G})+\sum_{C \in \mathscr{C}_{G}}\frac{|V(C)|}{2}-c(G_{1})-1\\
&=&\alpha(T_{G})+\sum_{C \in \mathscr{C}_{G}}\frac{|V(C)|}{2}-c(G).
\end{eqnarray*}

This completes the proof.
\end{proof}

\noindent
{\bf The proof of Theorem \ref{T50}.}
(Sufficiency.) We proceed by induction on the order $n$ of $(G, \sigma)$. If $n=1$, then the result holds trivially. Therefore we assume that $(G, \sigma)$ is a signed graph with order $n \geq 2$ and satisfies (i)-(iii). Suppose that any signed graph of order smaller than $n$ which satisfes (i)-(iii) is
lower-optimal. Since the cycles (if any) of $(G, \sigma)$ are pairwise vertex-disjoint, $(G, \sigma)$ has
exactly $c(G)$ cycles, i.e., $|\mathscr{O}_{G}|=c(G)$.

If $E(T_{G})=0$, i.e., $T_{G}$ is an empty graph, then each component of $(G, \sigma)$ is a cycle or an isolated vertex. By (ii) and Lemma \ref{L50}, we have $(G, \sigma)$ is lower-optimal.

If $E(T_{G}) \geq 1$. Then $T_{G}$ contains at least one pendant vertex. By (iii), one has that
$$\alpha(T_{G})=\alpha([T_{G}])+c(G)=\alpha(T_{G}-\mathscr{O}_{G})+c(G).$$
Thus, by Lemma \ref{L055}, there exists a pendent vertex of $T_{G}$ not in $\mathscr{O}_{G}$. Then, $(G, \sigma)$ contains at least one pendant vertex, say $u$. Let $v$ be the unique neighbour of $u$ and let
$(G_{0}, \sigma)=(G, \sigma)-\{ u, v \}$. It is obviously that $u$ is a pendant vertex of $T_{G}$ adjacent to $v$ and
$T_{G_{0}}=T_{G}-\{ u, v \}$. By Lemma \ref{L053}, one has that
$$\alpha(T_{G})=\alpha(T_{G}-v)=\alpha(T_{G}-\{ u, v \})+1.$$

{\bf Claim.} $v$ does not lie on any cycle of $(G, \sigma)$.

By contradiction, assume that $v$ lies on a cycle of $(G, \sigma)$. Then $v$ is in $\mathscr{O}_{G}$. Note that the size of $\mathscr{O}_{G}$ is $c(G)$. Then, $H:=(T_{G}-v) \cup K_{1}$ is a spanning subgraph of $T_{G}$.
Delete all the edges $e$ in $H$ such that $e$ contains at least one end-vertex in $\mathscr{O}_{G} \backslash \{ v \}$.
Thus, the resulting graph is $[T_{G}] \cup c(G)K_{1}$. By Lemma \ref{L054}, one has that
$$\alpha([T_{G}] \cup c(G)K_{1}) \geq \alpha((T_{G}-v) \cup K_{1}),$$
that is,
$$\alpha([T_{G}])+c(G) \geq \alpha(T_{G}-v)+1.$$
Then, we have
$$\alpha([T_{G}]) \geq \alpha(T_{G}-v)+1-c(G)=\alpha(T_{G})+1-c(G),$$
a contradiction to (iii). This completes the proof of the claim.

Thus, $v$ does not lie on any cycle of $(G, \sigma)$. Moreover, $u$ is also a pendant vertex of $[T_{G}]$ which adjacent to $v$ and $[T_{G_{0}}]=T_{G}-\{ u, v \}$. By Lemma \ref{L053}, one has that
$$\alpha([T_{G}])=\alpha([T_{G_{0}}])+1.$$
It is routine to checked that
$$c(G)=c(G_{0}).$$
Thus,
\begin{eqnarray*}
\alpha(T_{G_{0}})&=&\alpha(T_{G})-1\\
&=&\alpha([T_{G}])+c(G)-1\\
&=&\alpha([T_{G_{0}}])+1+c(G)-1\\
&=&\alpha([T_{G_{0}}])+c(G_{0}).
\end{eqnarray*}

Combining the fact that all cycles of $(G, \sigma)$ belong to $(G_{0}, \sigma)$, one has that
$(G_{0}, \sigma)$ satisfies all the conditions (i)-(iii). By induction hypothesis, we have $(G_{0}, \sigma)$
is lower-optimal. By Lemma \ref{L56}, we have $(G, \sigma)$ is lower-optimal.

(Necessity.) Let $(G, \sigma)$ be a lower-optimal signed graph. If $(G, \sigma)$ is a signed acyclic graph, then
(i)-(iii) holds directly. So one can suppose that $(G, \sigma)$ contains cycles. By Lemma \ref{L58}, one has that the cycles (if any) of $(G, \sigma)$ are pairwise vertex-disjoint and for each cycle $C_{l}$ of $(G, \sigma)$, either $ l \equiv 0 \ (\rm{mod} \ 4)$ and $\sigma(C_{l})=+$ or $ l\equiv 2 \ (\rm{mod} \ 4)$ and $\sigma(C_{l})=-$. This completes the proof of (i) and (ii).

Next, we argue by induction on the order $n$ of $(G, \sigma)$ to show (iii). Since $(G, \sigma)$ contains cycles, $n \geq 3$. If $n=3$, then $(G, \sigma)$ is a 3-cycle and (iii) holds trivially. Therefore we assume that $(G, \sigma)$ is a lower-optimal signed graph with order $n \geq 4$.
Suppose that (iii) holds for all lower-optimal signed graphs of order smaller than $n$.

If $E(T_{G})=0$, i.e., $T_{G}$ is an empty graph, then each component of $(G, \sigma)$ is a cycle or an isolated vertex. Then, (iii) follows immediately by Lemma \ref{L50}.

If $E(T_{G}) \geq 1$. Then $T_{G}$ contains at least one pendant vertex, say $x$. If $x$ is also a pendant vertex in $(G, \sigma)$, then $(G, \sigma)$ contains a pendant vertex. If $x$ is a vertex obtained by contracting a cycle of $(G, \sigma)$, then $(G, \sigma)$ contains a pendant cycle.
Then we will deal with (iii) with the following two cases.

{\bf Case 1.} $x$ is a pendant vertex of $(G, \sigma)$.

Let $y$ be the unique neighbour of $x$ and
$(G_{1}, \sigma)= (G, \sigma)-\{ x, y \}$. By Lemma \ref{L56}, one has that $y$ is not on any cycle of $(G, \sigma)$
and $(G_{1}, \sigma)$ is lower-optimal. By induction hypothesis, we have
$$\alpha(T_{G_{1}})=\alpha([T_{G_{1}}])+c(G_{1}).$$
By the fact that $y$ is not on any cycle of $(G, \sigma)$, then
$$c(G)=c(G_{1}).$$
Note that $x$ is also a pendant vertex of $T_{G}$ which adjacent to $y$, then $T_{G_{1}}=T_{G}-\{ x, y \}$ and
$[T_{G_{1}}]=[T_{G}]-\{ x, y \}$. By Lemma \ref{L053}, one has that
$$\alpha(T_{G})=\alpha(T_{G_{1}})+1, \alpha([T_{G}])=\alpha([T_{G_{1}}])+1.$$
Thus, we have
\begin{eqnarray*}
\alpha(T_{G})&=&\alpha(T_{G_{1}})+1\\
&=&\alpha([T_{G_{1}}])+c(G_{1})+1\\
&=&\alpha([T_{G}])-1+c(G)+1\\
&=&\alpha([T_{G}])+c(G).
\end{eqnarray*}

The result follows.

{\bf Case 2.} $(G, \sigma)$ contains a pendant cycle, say $C_{q}$.

Let $u$ be the unique vertex of $C_{q}$ of degree 3 and $K=G-C_{q}$. By Lemma \ref{L55}, one has that
$(K, \sigma)$ is lower-optimal. By induction hypothesis, we have
\begin{equation} \label{E20}
\alpha(T_{K})=\alpha([T_{K}])+c(K).
\end{equation}
In view of Lemma \ref{L55}, one has
\begin{equation} \label{E21}
\alpha(G)=\alpha(K)+\frac{q}{2}.
\end{equation}
It is obviously that
$$\mathscr{C}_{G}=\mathscr{C}_{K} \cup C_{q}.$$
Then, we have
\begin{equation} \label{E22}
\sum_{C \in \mathscr{C}_{G}}\frac{|V(C)|}{2}=\sum_{C \in \mathscr{C}_{K}}\frac{|V(C)|}{2}+\frac{q}{2}.
\end{equation}
Since  $(G, \sigma)$ and  $(K, \sigma)$ are lower-optimal, by Lemma \ref{L58}, we have
\begin{equation} \label{E23}
\alpha(T_{G})=\alpha(G)-\sum_{C \in \mathscr{C}_{G}}\frac{|V(C)|}{2}+c(G),
\end{equation}
and
\begin{equation} \label{E24}
\alpha(T_{K})=\alpha(K)-\sum_{C \in \mathscr{C}_{K}}\frac{|V(C)|}{2}+c(K).
\end{equation}
Moreover, it is routine to check that
\begin{equation} \label{E25}
c(G)=c(K)+1.
\end{equation}
Combining (\ref{E20})-(\ref{E25}), we have
\begin{eqnarray*}
\alpha(T_{G})&=&\alpha(G)-\sum_{C \in \mathscr{C}_{G}}\frac{|V(C)|}{2}+c(G)\\
&=&\alpha(K)+\frac{q}{2}-\sum_{C \in \mathscr{C}_{G}}\frac{|V(C)|}{2}+c(G)\\
&=&\alpha(K)-\sum_{C \in \mathscr{C}_{K}}\frac{|V(C)|}{2}+c(G)\\
&=&\alpha(K)-\sum_{C \in \mathscr{C}_{K}}\frac{|V(C)|}{2}+c(K)+1\\
&=&\alpha(T_{K})+1.
\end{eqnarray*}
Note that
\begin{equation} \label{E26}
[T_{G}]\cong [T_{K}].
\end{equation}
Then, in view of (\ref{E20}) and (\ref{E25})-(\ref{E26}), one has that
\begin{eqnarray*}
\alpha(T_{G})&=&\alpha(T_{K})+1\\
&=&\alpha([T_{K}])+c(K)+1\\
&=&\alpha([T_{G}])+c(G).\\
\end{eqnarray*}

This completes the proof.
$\square$

\section*{Acknowledgments}

This work was supported by the National Natural Science Foundation of China
(No. 11731002, 11771039 and 11771443), the Fundamental Research Funds for the Central Universities (No. 2016JBZ012) and the 111 Project of China (B16002).

\end{document}